\documentclass[reqno,12pt,draft]{amsart}
\usepackage[margin=1in]{geometry}
\usepackage{amssymb,upref,cite,enumerate}

\newcommand{\bb}{\mathbb}
\newcommand{\bbN}{{\mathbb{N}}}
\newcommand{\bbR}{{\mathbb{R}}}

\newcommand{\bbZ}{{\mathbb{Z}}}
\newcommand{\bbC}{{\mathbb{C}}}

\newcommand{\calH}{{\mathcal H}}

\newcommand{\calJ}{{\mathcal J}}
\newcommand{\calK}{{\mathcal K}}
\newcommand{\calL}{{\mathcal L}}

\newcommand{\no}{\nonumber}
\newcommand{\lb}{\label}
\newcommand{\f}{\frac}

\newcommand{\ol}{\overline}

\newcommand{\la}{\lambda}
\newcommand{\al}{\alpha}

\newcommand{\eps}{\varepsilon}

\newcommand{\rank}{\text{\rm{rank}}}

\newcommand{\bi}{\bibitem}

\newcommand{\beq}{\begin{equation}}
\newcommand{\eeq}{\end{equation}}
\newcommand{\ba}{\begin{align}}
\newcommand{\ea}{\end{align}}

\renewcommand{\Re}{\text{\rm Re}}
\renewcommand{\Im}{\text{\rm Im}}
\renewcommand{\ln}{\text{\rm ln}}

\DeclareMathOperator{\tr}{tr}

\allowdisplaybreaks
\numberwithin{equation}{section}

\newtheorem{theorem}{Theorem}[section]
\newtheorem{lemma}[theorem]{Lemma}

\theoremstyle{definition}
\newtheorem{definition}[theorem]{Definition}
\newtheorem{example}[theorem]{Example}

\theoremstyle{remark}
\newtheorem{remark}[theorem]{Remark}

\begin{document}
\title[Perturbation determinant]{On a perturbation determinant for accumulative operators}

\author[K.\ A.\ Makarov]{Konstantin A.\ Makarov}
\address{K.\ A.\ Makarov, Department of Mathematics, University of Missouri, Columbia, MO 65211, USA}
\email{makarovk@missouri.edu}

\author[A.\ Skripka]{Anna Skripka$^*$}
\address{A.\ Skripka, Department of Mathematics and Statistics,
	University of New Mexico, Albuquerque, NM 87131, USA}
\email{skripka@math.unm.edu}

\author[M.\ Zinchenko]{Maxim Zinchenko$^\dagger$}
\address{M.\ Zinchenko, Department of Mathematics and Statistics,
	University of New Mexico, Albuquerque, NM 87131, USA}
\email{maxim@math.unm.edu}

\thanks{\footnotesize ${}^*$Research supported in part by National Science Foundation grant DMS--1249186.
\\
${}^\dagger$Research supported in part by Simons Foundation grant CGM--281971.}



\begin{abstract}
For a purely imaginary sign-definite perturbation of a self-adjoint operator, we obtain exponential representations for the perturbation determinant in both upper and lower half-planes  and derive  respective trace formulas.
\end{abstract}

\maketitle



\section{Introduction}

The main goal of this note is to obtain new exponential representations for the perturbation determinant associated with a purely imaginary sign-definite perturbation of a self-adjoint operator.

Let $\bbC_\pm=\{z\in\bbC \,|\, \Im(z) \gtrless 0\}$. The starting point of our consideration is the exponential representation
(see \cite[Lemma 6.5]{NM}),
\begin{equation}
\label{introdet}
{\det}_{H/H_0}(z)=\exp\bigg (\frac{1}{\pi i} \int \frac{\zeta(\lambda)}{\lambda- z}\,d\lambda\bigg),\quad z\in\bbC_+,
\end{equation}
for the  perturbation determinant ${\det}_{H/H_0}(z)=\det((H-z)(H_0-z)^{-1})$ associated with a self-adjoint operator $H_0$ and an accumulative operator $H=H_0-iV$, where $V\geq 0$ is an element of the trace class $S^1$. Here the  nonnegative function $\zeta\in L^1(\bbR)$ is given by
\begin{equation*}
\zeta(\lambda)=\lim_{\varepsilon\to 0^+} \log \big
|{\det}_{H/H_0}(\lambda+i\varepsilon)\big|\;\text{ for a.e. } \lambda\in \bbR.
\end{equation*}
(See also Theorem~6.6 as well as Lemma~5.6 and Theorem~5.7 in \cite{NM} for general additive and some singular non-additive perturbation results, respectively.)
In Theorem \ref{uhpthm}, we give a new proof of \eqref{introdet}
and, in Theorem \ref{lhpl}, we obtain a complementary exponential representation for ${\det}_{H/H_0}(z)$ in $\bbC_-$.

Introducing the spectral shift function  $\xi(\lambda)$  via the boundary value of the argument of the perturbation determinant
\[
\xi(\lambda)=\frac1\pi\lim_{\varepsilon\to 0^+}\arg({\det}_{H/H_0}(\lambda+i\varepsilon)) \;\text{ for a.e. } \lambda\in \bbR,
\]
we show (see Theorem \ref{notl1}) that  $\xi $ is {\it never} integrable whenever $H\neq H^*$; in fact, $\xi$ is not even in $L^1_{w,0}(\bbR)$ (see \eqref{L1w0} for the definition of the weak zero space), but instead $\xi\in L^1_{w,0}(\bbR;\frac{d\la}{1+\la^2})$. By switching from Lebesgue integration to integration of type (A), we reconstruct the perturbation determinant from $\xi$ in $\bbC_+$ (with formula mimicking the self-adjoint case) in Theorem \ref{logdet}.

Using the aforementioned representations for the perturbation determinant,
in Theorem \ref{thmTF}, we obtain a trace formula for rational functions vanishing at infinity with poles in both $\bbC_-$ and $\bbC_+$, which is an analog of a trace formula for contractions derived in \cite{R87,R89,R94}.
Our approach to accumulative operators is based on complex and harmonic analysis, but, as distinct from \cite{R87,R89,R94}, does not appeal to functional models of accumulative operators.

For the history (up to 1990) of perturbation determinants and associated trace formulas in the non-self-adjoint setting we refer to \cite{AdN90} where contributions to the field by
H. Langer \cite{L65},
L.~A.~Sahnovi\v{c} \cite{Sah68},
R.~V.~Akopjan \cite{Ak73,Ak83},
P.~Jonas \cite{J86, J87},
V.~A.~Adamjan and B.~S.~Pavlov \cite{AdP79},
A.~V.~Rybkin \cite{R84,R87,R89},
M.~G.~Krein \cite{Kr89},
H.~Neidhardt \cite{N87,N88}
are discussed in detail; for recent developments see \cite{R94,R96,dsD,NM}. References to partial results for accumulative operators are also given in Remark \ref{partial}. We also remark that various concepts of generalized integration, including the Kolmogorov-Titchmarsh $A$-integral, appeared to be rather useful in harmonic analysis
\cite{T29}, probability theory \cite{K33}, as well as in perturbation theory for non-self-adjoint operators. In particular, the concept of $A$-integral has been systematically used for trace formulas associated with contractive trace class and special cases of Hilbert--Schmidt perturbations of a unitary operator (see \cite{R87,R89,R94,R96} and the references therein).

We recall that a closed densely defined operator $A$ is called accumulative if $\Im\left<Ah,h\right>\leq 0$ for every $h$ in the domain of $A$. Throughout the paper, we assume that every accumulative operator is maximal, which guarantees that $\bbC_+$ is a subset of the resolvent set of $A$ \cite{Ph59,Kr89}.

\section{Herglotz and outer functions}

We recall the canonical inner-outer factorization  theorem for the Hardy \\classes
$H_p$, $0<p\le \infty$, in the upper half-plane \cite[Chapter~VI:C, p.~119]{Ko08}.

\begin{theorem}\label{factth}
If $0\not\equiv F\in H_p(\bbC_+)$, $0<p\le \infty$, then
\[F(z)=I_F(z)\cdot O_F(z), \quad z\in\bbC_+,\]
where
\begin{enumerate}
\item $I_F$ is the inner factor of $F$ given by
\[I_F(z)=e^{i\gamma+i\alpha z}B(z)\exp\left(\frac{i}{\pi}\int_\bbR
\left(\frac{1}{\lambda-z}-\frac{\lambda}{1+\lambda^2}\right)d\mu_{\text{sing}}(\lambda)\right),\]
with
\begin{enumerate}
\item $\gamma\in\bbR$, $\alpha\ge 0$,

\item  a Blaschke product
\begin{equation}\label{blaschke}
B(z)=\prod_{k=1}^\infty\left(e^{i\alpha_k}\frac{z-z_k}{z-\overline{z_k}}\right),
\end{equation}
where
$z_k$ are the zeros of $F(z)$ in $\bbC_+$ and $\alpha_k\in\bbR$ are chosen so that $e^{i\alpha_k}\frac{i-z_k}{i-\overline{z_k}}\ge 0$,

\item $\mu_{\text{sing}}\ge0$ a singular measure on $\bbR$ satisfying
$\int_\bbR\frac{d\mu_{\text{sing}}(\lambda)}{1+\lambda^2} <\infty$,
\end{enumerate}

\item $O_F$ is the outer factor of $F$ given by
\begin{align}\lb{OuterFn}
O_F(z)=\exp \left(\frac{1}{\pi i}\int_\bbR \left(\frac{1}{\lambda-z}-\frac{\lambda}{1+\lambda^2}\right)\log |F(\lambda+i0)|\,d\lambda\right).
\end{align}
\end{enumerate}
\end{theorem}

\begin{remark}\label{remcon}
(i) We have the Blaschke condition \cite[Chapter~VI:C]{Ko08}
\begin{align}
\label{BC}
\sum_{k=1}^\infty\frac{\Im(z_k)}{|z-z_k|^2}<\infty,\quad z\in\bbC\setminus(\bbR\cup\{z_k\}_{k=1}^\infty).
\end{align}
(ii) If, in addition, $|F(z)|\le 1$, for $z\in\bbC_+$, then $F$ can be factorized as
$$
F(z)=B(z)\exp\bigg (\frac{i}{\pi}M(z)\bigg ),
$$
where $B(z)$ is the Blaschke product \eqref{blaschke} and $M(z)$ is the Herglotz function
$$
M(z)=\pi\alpha z +\pi\gamma +\int_\bbR \left(\frac{1}{\lambda-z}-\frac{\lambda}{1+\lambda^2}\right) d\mu(\lambda),\quad \mu\ge 0.
$$
(See, e.g., \cite[Chapter~VI, Section~59, Theorem~2]{Akhiezer} for representations of Herglotz functions.)
\end{remark}

\begin{definition}
We say that $F$ is an outer function if $F$ is analytic on $\bbC_+$, $|F|$ has finite boundary values a.e.\ on $\bbR$, $\log |F|\in L^1(\bbR;\frac{d\lambda}{1+\lambda^2})$, and for some $\theta \in \bbR$,
$$
F(z)=e^{i\theta} \, O_F(z), \quad z\in\bbC_+,
$$
where $O_F(z)$ is given by \eqref{OuterFn}.
\end{definition}

\begin{theorem}\label{outer} If $M(z)$ is a Herglotz function,
 then the function $1 -iM(z)$ is outer in $\bbC_+$.
\end{theorem}
\begin{proof}
The $H^\infty$-function  $F(z)=(1 -iM(z))^{-1}$ has non-negative real part in the upper half-plane $\bbC_+$. By \cite[Corollary 4.8 (a)]{G}, $F(z)$ is an outer function, so is  $F^{-1}(z)=1 -iM(z)$.
\end{proof}

\begin{proof}[Second proof]
 Since the function
$(1-iM(z))^{-1}$ is an analytic   contractive function with no zeros
in the upper half-plane
 by Remark \ref{remcon}, we have the representation
$$
(1-iM(z))^{-1}=\exp\bigg (\frac{i}{\pi}N(z)\bigg ),
$$
where $N(z)$ is a Herglotz function.
%
Next, the function $i(1-iM(z))$ is also Herglotz. Therefore, by the Aronszajn-Donoghue exponential Herglotz representation theorem (see, e.g., \cite[Theorem~2.4]{GT}),
$$
i(1-iM(z))=\exp \bigg (\frac{1}{\pi}L(z)\bigg )
$$
for some  absolutely continuously represented Herglotz function  $L(z)$ without the linear term.
Hence,
$$
N(z)-iL(z)= 2k\pi^2+\frac{\pi^2}{2},\quad \text{ for some } k\in \bbZ.
$$
Since $L$ has no linear term, so does $N$. Applying a variant of the  brothers Riesz's theorem
for  the upper  half-plane that states that
if a complex-valued finite Borel measure $\mu$ on $\bbR$ satisfies
$$
\int_\bbR \frac{1+z\lambda}{\lambda-z}\,d\mu(\lambda)=Az+B,
$$
for all $z\in\bbC_+$ and some $A, B\in \bbC$, then $A=0$ and $\mu$ is absolutely continuous,
yields the representation
 $$N(z)=\gamma+\int\left (\frac{1}{\lambda-z}-\frac{\lambda}{1+\lambda^2}
\right )d\nu(\lambda),
$$
for some $\gamma\in \bbR$ and some absolutely continuous measure $\nu$ such that
$$
\int_\bbR\frac{d\nu(\lambda)}{1+\lambda^2}<\infty.
$$
Now, the exponential representation
$$
(1-iM(z))=\exp\bigg (-\frac{i}{\pi}N(z)\bigg )
$$
shows that
$
1-iM(z)$
is an outer function.
\end{proof}

\section{Exponential representation for the perturbation determinant.}

The main goal of this section is to obtain representations for the perturbation determinant associated with an accumulative trace class perturbation of a self-adjoint operator.  As distinct from perturbation theory for self-adjoint operators, initial exponential  representations for the perturbation determinant appear to be quite different in $\bbC_-$ and $\bbC_+$.

We start with the case of the upper half-plane and show that the perturbation determinant is an outer function in $\bbC_+$ by reducing the general case to the case of rank-one perturbations and obtain an exponential representation for it.

\begin{lemma}\label{l1} Let $H_0$ be a maximal accumulative operator, $\alpha >0$, $P$ a one-dimensional orthogonal  projection, and let $H=H_0-i\alpha  P$. Then, the perturbation determinant ${\det}_{H/H_0}(z)$ is an outer function in the upper half-plane.
Moreover,
$$
{\det}_{H/H_0}(z)=\exp\bigg (\frac{1}{\pi i} \int \frac{\zeta(\lambda)}{\lambda-z}\,d\lambda\bigg), \quad z\in\bbC_+,
$$
where $\zeta\in L^1(\bbR)$ is given by
\begin{equation}\label{lapunov}
\zeta(\lambda)=\lim_{\varepsilon\to 0^+} \log
|{\det}_{H/H_0}(\lambda+i\varepsilon)|\ge 0,\quad  \text {a.e.}\,\, \lambda\in \bbR,
\end{equation}
and
\begin{equation}\label{norma}
\| \zeta\|_{L^1}=\pi \alpha.
\end{equation}

\end{lemma}
\begin{proof}
Suppose that for every $g\in\calH$,
$$
Pg=\left<g, f\right> f, \quad \|f\|=1.
$$
Then for all $z\in \rho(H_0)$,
\begin{align}
\label{rank1}
{\det}_{H/H_0}(z)=\det(I-i\alpha P(H_0-z)^{-1})=1-i\alpha \left<(H_0-z)^{-1}f, f\right>.
\end{align}
Since $\alpha >0$ and $H_0$ is an accumulative operator, the quadratic form
$\alpha \left<(H_0-z)^{-1}f, f\right>$ is a Herglotz function in the upper half plane.
Indeed, denote by $\calL$ the minimal  self-adjoint dilation of the accumulative operator $H_0$ in a Hilbert space $\calK$, $\calH\subset \calK$ (see \cite{SFBK10} for details), so that
$$(H_0-z)^{-1}=P_{\calH}(\calL-z)^{-1}|_{\calH}, \quad z\in \bbC_+.$$
Hence, $\big<(H_0-z)^{-1}f,f\big>=\big<(\calL-z)^{-1}\tilde f,\tilde f\big>$ is a Herglotz function.
Here $\tilde f=\calJ f$, with $\calJ:\calH \to \calK$ the natural imbedding of the Hilbert space $\calH$ into the Hilbert space $\calK$.

By  Theorem \ref{outer} and the representation \eqref{rank1}, the perturbation determinant
${\det}_{H/H_0}(z)$ is an outer function in the upper
half-plane. Therefore
$$
{\det}_{H/H_0}(z)=e^{i\gamma}\exp\bigg (\frac{1}{\pi i} \int
\zeta(\lambda)\bigg(\frac{1}{\lambda- z}-\frac{\lambda}{1+\lambda^2}\bigg)\,d\lambda\bigg), \quad z\in\bbC_+,
$$
where $\gamma\in\bbR$, the function $\zeta(\lambda)$ is given by \eqref{lapunov}, and
$$
\int \frac{|\zeta(\lambda)|}{1+\lambda^2}\,d\lambda<\infty.
$$
Since
$$
\Re ({\det}_{H/H_0}(z))=1-\Re\left(i\alpha \left<(H_0-z)^{-1}f, f\right>\right)\ge 1, \quad z\in\bbC_+,
$$
and hence,
$$|{\det}_{H/H_0}(z)|\ge 1, \quad z\in\bbC_+, $$
the function  $\zeta$  given by \eqref{lapunov} is non-negative
almost everywhere.
We also have the representation
$$
{\det}_{H/H_0}(z)=1-i\alpha \big<(\calL-z)^{-1}\tilde f, \tilde f\big>, \quad z\in \bbC_+.
$$
Since $\big<E_{\mathcal{L}}(\cdot)\tilde f, \tilde f\big>$ is a finite measure, where $E_{\mathcal{L}}$ is the spectral measure of $\mathcal{L}$, we have the asymptotics
\begin{equation}\label{detas}
{\det}_{H/H_0}(iy)=1+\frac{\alpha}{y} +o(y^{-1}) \;\text{ as }\; y\to
+\infty.
\end{equation}
Hence,
$$
\sup_{y>0}\left |\frac{y}{\pi i} \int \zeta(\lambda)\left(\frac{1}{\lambda-
iy}-\frac{\lambda}{1+\lambda^2}\right)\,d\lambda\right |<\infty,
$$
which proves (see, e.g., \cite[Chapter~VI, Section~59, Theorem~3]{Akhiezer}) that $\zeta$ is an integrable function and, therefore, the perturbation determinant admits the representation
$$
{\det}_{H/H_0}(z)=\exp\bigg (\frac{1}{\pi i} \int\frac{\zeta(\lambda)}{\lambda-z}\,d\lambda
\bigg), \quad z\in\bbC_+.
$$
One then observes that
\begin{equation}\label{detsa}
{\det}_{H/H_0}(iy)=\exp \bigg (\frac{1}{\pi y} \int_\bbR
\zeta(\lambda)\,d\lambda+o(y^{-1})\bigg ) \;\text{ as }\; y\to
+\infty.
\end{equation}
 Comparing \eqref{detas} and \eqref{detsa} yields
 $$
 \int_\bbR \zeta(\lambda)\,d\lambda=\pi\alpha,
$$
which proves \eqref{norma}, since $\zeta $ is non-negative a.e.
\end{proof}

\begin{theorem}\label{uhpthm} Let $H_0$ be a maximal accumulative operator, $0\leq V=V^*\in S^1$, and let $H=H_0-iV$. Then, the perturbation determinant ${\det}_{H/H_0}(z)$ is an outer function in $\bbC_+$. Moreover,
\begin{equation}\label{ExpRepZeta}
{\det}_{H/H_0}(z)=\exp\bigg (\frac{1}{\pi i} \int \frac{\zeta(\lambda)}{\lambda- z}\,d\lambda\bigg), \quad z\in\bbC_+,
\end{equation}
where
\begin{equation}\label{zeta}
\zeta(\lambda)=\lim_{\varepsilon\to 0^+} \log \big
|{\det}_{H/H_0}(\lambda+i\varepsilon)\big|\ge 0\quad\text{\rm a.e. } \lambda\in \bbR,
\end{equation}
with
$$
\|\zeta\|_{L^1(\bbR)}=\tr (V).
$$
\end{theorem}

\begin{proof}
Let $V=\sum_{k=1}^\infty \alpha_k P_k$
be the spectral decomposition of the trace class operator $V$,
where $P_k$, $  k=1,2, \dots$, are one-dimensional spectral projections and
$\alpha_1\ge \alpha_2\ge \dots$,
are the corresponding eigenvalues  counting multiplicity.
Introducing the accumulative operators
$$
H_{k+1}=H_{k}-i\alpha_k P_k, \quad k\in \bbN,
$$
and taking into account the multiplicativity of the perturbation
determinant \cite[Theorem~3.5]{Simon}, one obtains
\begin{equation}\label{prod}
{\det}_{H_n/H_0}(z)=\prod_{k=1}^n{\det}_{H_{k}/H_{k-1}}(z), \quad z\in\bbC_+.
\end{equation}
By Lemma \ref{l1}, every factor in  the product \eqref{prod} is an outer function in $\bbC_+$, so is ${\det}_{H_n/H_0}(z)$, and, moreover, one has the representations
$$
{\det}_{H_n/H_0}(z)=\exp\bigg (\frac{1}{\pi i} \int \frac{\zeta_n(\lambda)}{\lambda- z}\,d\lambda\bigg), \quad z\in\bbC_+, \quad n\in \bbN,
$$
where $\{\zeta_k\}_{k\in \bbN}$ is a monotone sequence of nonnegative summable functions.
It also follows from Lemma  \ref{l1} that
$$
\int_{\bbR}\zeta_n(\lambda)\,d\lambda=\sum_{k=1}^n \alpha_k,\quad n\in \bbN.
$$
Since by hypothesis $V$ is a trace class operator, the series $\sum_{k=1}^\infty \alpha_k$ converges, and therefore, the sequence $\zeta_n$ converges pointwise a.e. and in the topology of the space $L^1(\bbR)$
to a summable function $\zeta$.
By \cite[Theorem~3.4]{Simon},
 $$
 \lim_{n\to \infty}{\det}_{H_n/H_0}(z)={\det}_{H/H_0}(z)
 $$
uniformly on compact subsets of $\bbC_+$ and
 $$
 \lim_{n\to \infty}\exp\bigg (\frac{1}{\pi i} \int \frac{\zeta_n(\lambda)}{\lambda- z}\,d\lambda\bigg)=\exp\bigg (\frac{1}{\pi i} \int \frac{\zeta(\lambda)}{\lambda- z}\,d\lambda\bigg).
 $$
Thus, one obtains  the representation
 $$
{\det}_{H/H_0}(z)= \exp\bigg (\frac{1}{\pi i} \int \frac{\zeta(\lambda)}{\lambda- z}\,d\lambda\bigg), \quad z\in\bbC_+.
 $$
In particular, the perturbation determinant ${\det}_{H/H_0}(z)$ is an outer function (in $\bbC_+$) and \eqref{zeta} holds.
\end{proof}

\begin{remark}
An analog of the representation \eqref{ExpRepZeta} with a measure not known to be absolutely continuous has appeared previously in \cite[Theorem~9.1]{Kr89}. Recently, \eqref{ExpRepZeta} was extended in \cite[Theorem~6.6]{NM} to pairs of maximally accumulative operators $H_0$ and $H$ with trace class differences by treating separately purely imaginary and purely real perturbations and using multiplicativity of the perturbation determinant. Further generalizations of \eqref{ExpRepZeta} can be found in \cite[Theorem~5.7]{NM}.

To obtain an exponential representation for ${\det}_{H/H_0}$ in $\bbC_-$, we remark that the Schwarz reflection principle, which was
valid in the self-adjoint setting, does not hold anymore, and it should be modified by the relation
\begin{equation}\label{wh}
{\det}_{H/H_0}(\lambda-i0)={\det}_{H/H^*}(\lambda-i0)\,\overline{{\det }_{H/H_0}(\lambda+i0)}, \quad \text{ a.e. } \lambda\in \bbR.
\end{equation}
\end{remark}

\begin{theorem}\label{lhpl}
Suppose that $H_0=H_0^*$, $0\leq V=V^*\in S^1$, and let $H=H_0-iV$.
Then, the perturbation determinant ${\det}_{H/H_0}(z)$, $z\in\bbC_-$, admits the representation
\begin{align} \label{dddd}
{\det}_{H/H_0}(z)=e^{i\gamma-iaz}B(z)&\exp
\bigg(\frac{1}{\pi i}\int_\bbR
\frac{1+\lambda z}{\lambda-z}\,d\mu(\lambda)\bigg) \no\\
&\times
\exp\bigg (-\frac{1}{\pi i} \int \frac{\zeta(\lambda)}{\lambda- z}\,d\lambda\bigg),
\end{align}
where $\gamma\in \bbR$, $a\ge 0$, $B(z)$ is the Blaschke product associated with the eigenvalues
of $H$ in $\bbC_-$, $0\le\mu$ is a finite Borel measure on $\bbR$, and $\zeta$ is the
summable function given by \eqref{zeta}.
\end{theorem}

\begin{proof}
As a consequence of the multiplication rule, we have the decomposition
\begin{align}\label{miltr}
\nonumber
{\det}_{H/H_0}(z)&={\det}_{H/H^*}(z)\cdot{\det}_{H^*/H_0}(z)\\
&={\det}_{H/H^*}(z)\cdot\overline{{\det}_{H/H_0}(\overline{z})},\quad z\in \rho(H^*)\cap\rho(H_0).
\end{align}
By Theorem \ref{uhpthm} (in accordance with \eqref{zeta}), we have
\begin{align}
\label{miltr1}
\overline{{\det}_{H/H_0}(\overline{z})}
=\exp\left (\overline{\frac{1}{\pi i} \int \frac{
\zeta(\lambda)}{\lambda- \overline{z}}\,d\lambda}\right)
=\exp\left (-\frac{1}{\pi i} \int \frac{
\zeta(\lambda)}{\lambda- z}\,d\lambda\right), \quad z\in\bbC_-.
\end{align}
It was established in \cite[eq.~(8.16)]{Kr89} that in the lower half-plane $\bbC_-$, the perturbation determinant ${\det}_{H/H^*}(z)$, $z\in\bbC_-,$ is an analytic contractive function. Thus, by the standard inner-outer factorization (see Theorem \ref{factth} and Remark \ref{remcon}),
\begin{equation}
\label{HH*}
{\det}_{H/H^*}(z)=e^{i\gamma-iaz}B(z)\exp\bigg(\frac{1}{\pi i}\int_\bbR \frac{1+\lambda z}{\lambda-z}
\,d\mu(\lambda)\bigg),  \quad z\in\bbC_-.
\end{equation}
Combining \eqref{miltr}--\eqref{HH*} completes the proof.
\end{proof}

\section{The argument of the perturbation determinant}

Let $H_0=H_0^*$ and let $0\leq V=V^*\in S^1$. Let $\{V_n\}_{n=1}^\infty$ be finite-rank approximations to $V$ with $V_n\geq 0$, $\rank(V_n)\leq n$, and $V_n\to V$ in the trace class norm. Denote $H=H_0-iV$ and $H_n=H_0-iV_n$.
Introducing the spectral shift functions $\xi_n(\lambda)$
associated with the pairs $H_n$ and $H_0$ by the standard relation
$$
\xi_n(\lambda)=\frac1\pi\arg ({\det}_{H_n/H_0}(\lambda+i0)), \quad \text{
a.e. } \lambda\in \bbR,
$$
by \eqref{rank1} and \eqref{prod}, one obtains the bounds
\begin{equation}\label{maxim}
-\frac n2\le \xi_n(\lambda)\le  \frac n2, \quad \text{
a.e. } \lambda\in \bbR, \quad n\in \bbN.
\end{equation}
Thus, in addition to \eqref{ExpRepZeta}, one also has the following exponential representation for all $z\in\bbC_+$,
\begin{align} \lb{DetLimXi}
{\det}_{H/H_0}(z)=|{\det}_{H/H_0}(i)|\exp\left  (\lim_{n\to \infty} \int_\bbR\xi_n(\lambda)\left  (\frac{1}{\lambda-
z}-\frac{\lambda}{1+\lambda^2}\right)\,d\lambda\right).
\end{align}
However, in general, one cannot bring the limit under the integral due to the fact that the limit of $\xi_n(\lambda)$,
\begin{align}\label{xi1}
\xi(\lambda)=\frac1\pi\arg ({\det}_{H/H_0}(\lambda+i0)), \quad \text{a.e. } \lambda\in \bbR,
\end{align}
can be non-locally integrable, and hence, not in $L^1(\bbR,\f{d\la}{1+\la^2})$ as discussed in Example \ref{notL1loc} (see also \cite[Ex.~3.10]{N87}). We remark that the membership  $\xi\in L^1(\bbR,\f{d\la}{1+\la^2})$ can be recovered if the perturbation is slightly stronger than the trace class (see,  e.g., \cite{N87}).
Nonetheless, the spectral shift function $\xi(\lambda)$ given by \eqref{xi1}
is an element of the larger space of A-integrable functions $(A)L^1(\bbR,\f{d\la}{1+\la^2})$ defined below.

\begin{definition}Let $(X,\mathcal{A},\mu)$ be a $\sigma$-finite measure space,  the space $L^1_{w,0}(X,\mu)$ consists of all measurable functions $f(x)$ that satisfy
\begin{align}\lb{L1w0}
\mu\{x:\,|f(x)|>t\}=o\left(\frac1t\right) \;\text{ as }\; t\rightarrow \infty \;\text{ and }\; t\rightarrow 0^+.
\end{align}
For finite measures $\mu$, the above condition as $t\to0^+$ is automatically satisfied. A function $f$ is said to be A-integrable (see, e.g., \cite{Al81}), if $f\in L^1_{w,0}(X,\mu)$ and the limit
\[(A)\int_X f(x)\,d\mu(x):=\lim_{\substack{b\rightarrow 0\\ B\rightarrow \infty}}\int_{\{y\in X:\,b\leq |f(y)|\leq B\}} f(x)\,d\mu(x)\] exists.
By definition, the space   $(A)L^1(X,\mu)$ consists of all $A$-integrable functions.
\end{definition}

By changing the type of integration in \eqref{DetLimXi} we obtain the following exponential representation for the perturbation determinant in terms of the spectral shift function $\xi(\la)$.

\begin{theorem} \label{logdet}
Let $H_0$ be a maximal accumulative operator, $0\leq V=V^*\in S^1$, and define $H=H_0-i V$. Then
\begin{equation}\label{Aint}
{\det}_{H/H_0}(z)
=|{\det}_{H/H_0}(i)|\exp\left ((A)\int_\bbR \frac{1+z\lambda}{\lambda-z}\xi(\lambda)\f{d\lambda}{1+\lambda^2}\right), \quad z\in\bbC_+.
\end{equation}
\end{theorem}

To avoid  confusion, we stress that the A-integral on the right hand-side of \eqref{Aint} should be understood as the A-integral of the function  $\frac{1+z\lambda}{\lambda-z}\xi(\lambda)$, considered as an element of the weighted space $(A)L^1(\bbR,\f{d\la}{1+\la^2})$.

The proof of the theorem is based on the following version of Herglotz-type A-integral formula for functions analytic in $\bbC_+$.

\begin{lemma} \label{AintLem}
If $f$ is analytic in $\bbC_+$ with boundary
values in $L^1_{w,0}(\bbR;\frac{d\lambda}{1+\lambda^2})$, then
$\Re f$ and $\Im f$ are $A$-integrable on $\bbR$ with respect to $\f{d\la}{1+\lambda^2}$ and
\begin{align}
f(z) &= i\Im f(i) + \frac{1}{\pi i} \, (A)\!\!\int_\bbR
\frac{1+\lambda z}{\lambda-z}\Re f(\lambda)\,\frac{d\lambda}{1+\lambda^2} \lb{Ath1}
\\
&= \Re f(i) + \frac{1}{\pi} \, (A)\!\!\int_\bbR
\frac{1+\lambda z}{\lambda-z}\Im f(\lambda)\,\frac{d\lambda}{1+\lambda^2}, \quad z\in\bbC_+. \lb{Ath2}
\end{align}
\end{lemma}

\begin{proof}
Let $F(z)=f(i\frac{1-z}{1+z})$. Then $F(z)$ is analytic on the unit
disk $\mathbb D$ and its boundary value is a function in $L^1_{w,0}(\partial\bb D)$. By Aleksandrov's theorem \cite[Theorem~2.3.6]{CMR06}, the function $F$ is $A$-integrable on $\partial\bb D$ with respect to the Lebesgue measure and
\begin{align}
F(z)=\frac{1}{2\pi i} \, (A)\!\!\int_{\partial\mathbb D} \frac{F(w)}{w-z}\,dw
= (A)\!\!\int_0^{2\pi} \frac{F(e^{i\theta})}{1-e^{-i\theta}z}
\,\frac{d\theta}{2\pi}, \quad z\in\mathbb D. \lb{Ath3}
\end{align}
In particular, applying \eqref{Ath3} to the function
$\frac{F(w)}{1-w\ol z}$ we get
\begin{align}
\ol{F(0)} = \ol{(A)\!\!\int_0^{2\pi}
\frac{F(e^{i\theta})}{1-e^{i\theta}\ol z} \,\frac{d\theta}{2\pi}}
= (A)\!\!\int_0^{2\pi} \frac{\ol{F(e^{i\theta})}}{1-e^{-i\theta}z}
\,\frac{d\theta}{2\pi}. \lb{Ath4}
\end{align}
Adding \eqref{Ath3} and \eqref{Ath4} yields,
\begin{align}
F(z)+\ol{F(0)} &= (A)\!\!\int_0^{2\pi} \frac{2\Re
F(e^{i\theta})}{1-e^{-i\theta}z} \,\frac{d\theta}{2\pi} \lb{Ath5}
\\ &=
(A)\!\!\int_0^{2\pi} \frac{1+e^{-i\theta}z}{1-e^{-i\theta}z} \Re
F(e^{i\theta}) \,\frac{d\theta}{2\pi} + (A)\!\!\int_0^{2\pi} \Re
F(e^{i\theta}) \,\frac{d\theta}{2\pi} \no
\\&=
(A)\!\!\int_0^{2\pi} \frac{1+e^{-i\theta}z}{1-e^{-i\theta}z} \Re
F(e^{i\theta}) \,\frac{d\theta}{2\pi} + \Re F(0), \quad z\in\mathbb D,
\lb{Ath6}
\end{align}
where \eqref{Ath5} with $z=0$ was used to evaluate the last integral. Thus,
\begin{align}
F(z)=i\Im F(0) + (A)\!\!\int_0^{2\pi}
\frac{e^{i\theta}+z}{e^{i\theta}-z} \,\Re F(e^{i\theta})
\,\frac{d\theta}{2\pi}, \quad z\in\mathbb D. \lb{Ath7}
\end{align}
Rewriting \eqref{Ath7} in terms of $f(z)$ and changing variables under
the integral yield \eqref{Ath1}. Replacing $f(z)$ by $if(z)$ in
\eqref{Ath1} gives \eqref{Ath2}.
\end{proof}

\begin{proof}[Proof of Theorem \ref{logdet}]
Let $f(z)=\log(\text{det}_{H/H_0}(z))$, then $\Re f(\la+i0)=\zeta(\la)$ and $\Im f(\la+i0)=\pi\xi(\la)$, $\la\in\bb R$. By Theorem~\ref{uhpthm}, the function $\zeta(\la)$ is in $L^1(\bb R)$ and hence $\frac{1+z\lambda}{\lambda-z}\Re f(\lambda+i0)$ is in $L^1(\bb R)\subset L^1_{w,0}(\bbR,\f{d\la}{1+\la^2})$, $\Im(z)\ne 0$. Moreover, since the spectral shift function  $\xi(\la)$ is the Hilbert transform of the $L^1(\bb R)$ function $\zeta(\la)$, it follows from \cite[(1.6)]{PSZ10} that $\xi(\la)$ satisfies \eqref{L1w0} as $t\to\infty$ and hence so does the function $\frac{1+z\lambda}{\lambda-z}\Im f(\lambda+i0)$, $\Im(z)\ne 0$. Since the measure $\f{d\la}{1+\la^2}$ is finite, it follows that $\frac{1+z\lambda}{\lambda-z}\Im f(\lambda+i0)$ is in $L^1_{w,0}(\bbR,\f{d\la}{1+\la^2})$, $\Im(z)\ne 0$. Thus, $f(z)$ satisfies the assumptions of Lemma~\ref{AintLem} and so \eqref{Aint} follows from \eqref{Ath2}.
\end{proof}

\section{A trace formula}

We will now derive a trace formula for rational functions that may have poles in both $\bbC_+$ and $\bbC_-$.
Denote \[\mathcal{F}=\text{span}\big\{\lambda\mapsto (\lambda-z)^{-k}:\, k\in\bbN,\, z\in \rho(H_0)\cap\rho(H)\cap(\bbC\setminus\bbR)\big\}.\]
Let $\mathcal{P_\pm}$ be the orthogonal projections onto the Hardy spaces $H^2_\pm(\bbR)$ \cite[Chapter~VI]{Ko08}.

\begin{theorem}
\label{thmTF}
Suppose that $H_0=H_0^*$, $0\le V=V^*\in S^1$, and let $H=H_0-iV$. Then,
\begin{align}
\label{trfimpr}
\nonumber
&\tr(f(H)-f(H_0))=
\sum_k  \left ((\mathcal{P}_+f)(z_k)-(\mathcal{P}_+f)(\overline{z_k})\right)
+(A)\int_{\bbR} f'(\lambda)\xi(\lambda)\,d\lambda
\\&\qquad
+\frac{1}{\pi i}\int_{\bbR}(\mathcal{P_+}f')(\lambda)(1+\lambda^2)\,d\mu(\lambda)
-ia\,\text{\rm Res}\, \vert_{w=\infty}\, (\mathcal{P}_+f)(w),
\end{align}
for $f\in\mathcal{F}$, where $a$, $\mu$ are as in \eqref{dddd} and $z_k$ are eigenvalues of $H$.
\end{theorem}

\begin{proof}
By the known argument (see, e.g., \cite[Chapter~IV, \S~3.2, Prop.~5]{GK}),
\begin{equation}\label{gentrace}
\tr\big((H-z)^{-1}-(H_0-z)^{-1}\big)=-\frac{\frac{d}{dz}{\det}_{H/H_0}(z)}{{\det}_{H/H_0}(z)},  \quad z\in
\rho(H)\cap\rho(H_0).
\end{equation}
Therefore, by Theorem \ref{uhpthm},
\begin{equation}\label{step1}
\tr\big((H-z)^{-1}-(H_0-z)^{-1}\big)=-\frac{1}{\pi i} \int_{\bbR}\frac{\zeta(\lambda)}{(\lambda- z)^2}\,d\lambda,\quad z\in\bbC_+.
\end{equation}
By the representation \eqref{gentrace} and Theorem \ref{lhpl},
\begin{align}
\label{tfla}
&\tr\big((H-z)^{-1}-(H_0-z)^{-1}\big)
=-\sum_k \bigg(\frac{1}{z-z_k}
-\frac{1}{z-\overline{z_k}}\bigg) \no\\
&\qquad
+\frac{1}{\pi i} \int_{\bbR} \frac{\zeta(\lambda)}{(\lambda- z)^2}\,d\lambda
-\frac{1}{\pi i} \int_{\bbR}\frac{1+\lambda^2}{(\lambda-z)^2}\,d\mu(\lambda)+ia, \quad z\in\rho(H)\cap\bbC_-.
\end{align}

Next we combine \eqref{step1} and \eqref{tfla} to obtain \eqref{trfimpr}. Denote
\begin{align}
\label{fz}
f_{z,p}(\lambda)=(\lambda-z)^{-p},\quad p\in\bbN,\; z\in\rho(H_0)\cap\rho(H)\cap(\bbC\setminus\bbR).
\end{align}
The series $\sum_k \left (\frac{1}{z-z_k}-\frac{1}{z-\overline{z_k}}\right )$ converges uniformly on every compact subset of $\rho(H_0)\cap\rho(H)\cap(\bbC\setminus\bbR)$. By differentiating \eqref{step1} and \eqref{tfla} with respect to $z$, we obtain
\begin{align}
\label{trf}
&\tr\big(f_{z,p}(H)-f_{z,p}(H_0)\big)=
\frac{1}{\pi i} \int_{\bbR}
f_{z,p}'(\lambda)\zeta(\lambda)\,d\lambda\times \begin{cases}
\,\,\,\,1,& z\in\bbC_+\\
-1,&z\in\bbC_-
\end{cases} \no\\
&\quad
+\left (ia+\frac{1}{\pi i}\int_{\bbR}f_{z,p}'(\lambda)(1+\lambda^2)\,d\mu(\lambda)+\sum_k  \big(f_{z,p}(z_k)-f_{z,p}(\overline{z_k})\big)\right) \no\\
&\qquad
\times \begin{cases}
0,& z\in\bbC_+\\
1,&z\in\bbC_-
\end{cases}
\end{align}
for every $p\in\bbN$, $z\in \rho(H_0)\cap\rho(H)\cap(\bbC\setminus\bbR)$, where $a$, $\mu$ are as in \eqref{dddd}.

Denote by $\mathcal{T}$ the Hilbert transform on $ L^2(\bbR)$, respectively, so that
$\mathcal{T}=\frac1i (\mathcal{P_+}-\mathcal{P_-})$.
Fix $z\in \rho(H_0)\cap\rho(H)\cap(\bbC\setminus\bbR)$ and $p\in\bbN$.
It is easy to see that for $f_{z,p}$ given by \eqref{fz},
$$\mathcal{P}_+f'_{z,p}=f'_{z,p}\times \begin{cases}
0, & z\in \bbC_+\\
1,&z\in \bbC_-
\end{cases}
$$
and that
$$\mathcal{T}f'_{z,p}=f'_{z,p}\times \begin{cases}
-1, & z\in \bbC_+\\
\,\,\,\,\,1,&z\in \bbC_-
\end{cases}.
$$
Hence, the trace formula \eqref{trf} can be rewritten  via the Hilbert transform:
\begin{align}
\label{trtrtr}
\nonumber
&\tr(f_{z,p}(H)-f_{z,p}(H_0))=
\sum_k  \left ((\mathcal{P}_+ f_{z,p})(z_k)-(\mathcal{P}_+ f_{z,p})(\overline{z_k})\right) \no\\
&\qquad
-\frac{1}{\pi}\int_{\bbR} (\mathcal{T}f'_{z,p})(\lambda)
\zeta(\lambda)\,d\lambda
+\frac{1}{\pi i}\int_{\bbR}(\mathcal{P_+}f'_{z,p})(\lambda)(1+\lambda^2)\,d\mu(\lambda)
\\&\qquad
-ia\,\text{Res}\, \vert_{w=\infty}\,
\left (\mathcal{P}_+f_{z,p}(w)\right).\nonumber
\end{align}
It is proved  in \cite{Anter} that if $\phi\in L^p(\bbR)\cap L^\infty(\bbR)$, $p\geq 1$, with $\mathcal{T}(\phi)\in L^\infty(\bbR)$, and $h\in L^1(\bbR)$,
\[\int_\bbR h(x)(\mathcal{T}\phi)(x)\,dx=-(A)\int_\bbR(\mathcal{T}h)(x)\phi(x)\,dx\]
(see also \cite{U2} for the analogous result on the unit circle),
which along with \eqref{trtrtr} gives us \eqref{trfimpr} for $f=f_{z,p}$. Taking linear combinations of functions $f_{z,p}$ extends \eqref{trfimpr} to all $f\in\mathcal{F}$.
\end{proof}

\begin{remark}
\label{partial}
(i)
If $H$ and $H_0$ is a pair of self-adjoint operators with $H-H_0\in S^1$, then the analog of \eqref{trfimpr} has a simpler form:
\begin{equation}
\label{classics}
\tr\big(f(H)-f(H_0)\big)=\int_\bbR f'(\lambda)\xi(\lambda)\,d\lambda,
\end{equation}
as established in \cite{Krein}. A detailed list of references on \eqref{classics} can be found in the surveys \cite{BY,S}; references on higher order trace formulas can be found in \cite{S}. Attempts to extend the trace formula \eqref{classics} to accumulative operators $H_0$ and $H$ resulted in consideration of only selected pairs of accumulative $H_0$ and $H$ and led to modification of either the left or right hand side of  \eqref{classics} 
\cite{Sah68,AdP79,R84,N87,N88,Kr89,AdN90}. It is also known \cite{dsD} that for every pair of maximal accumulative operators $H_0$ and $H$, with $H-H_0\in S^1$, there exists a finite measure $\mu$ such that $\|\mu\|\leq\|V\|_1$ and
\begin{align*}
&\tr\big(f(H)-f(H_0)\big)=\int_\bbR f'(\lambda)\,d\mu(\lambda),\\
&f\in\text{span}\{\lambda\mapsto (z-\lambda)^{-k}:\, k\in\bbN,\, z\in\bbC_+\}.
\end{align*}

(ii)
By adjusting the reasoning in the proof of Theorem \ref{thmTF} to the perturbation determinant ${\det}_{H/H^*}(z)$, we obtain that for $H_0=H_0^*$, $0\le V=V^*\in S^1$, and  $H=H_0-iV$,
\begin{align}
\label{trfimpr*}
\nonumber
&\tr(f(H)-f(H^*))=
\sum_k \big(f(z_k)-f(\overline{z_k})\big)
\\&\qquad
+\frac{1}{\pi i}\int_{\bbR}f'(\lambda)(1+\lambda^2)\,d\mu(\lambda)
-ia\,\text{\rm Res}\, \vert_{w=\infty}\,(f(w)),
\end{align}
for rational functions $f\in C_0(\bbR)$ with poles in $\rho(H)\cap\rho(H^*)\cap\bbC_-$, where $a$, $\mu$ are as in \eqref{dddd}. By taking complex conjugation in \eqref{trfimpr*}, we extend the formula to all rational functions $f\in C_0(\bbR)$ with poles in $\rho(H)\cap\rho(H^*)$. The formula \eqref{trfimpr*} was obtained in \cite[Theorem~1]{AdP79} using a functional model of accumulative operators and in \cite[Theorem~8.3 and 8.4]{Kr89} via the perturbation determinant.
A similar formula for bounded dissipative operators with absolutely continuous spectrum was obtained earlier in \cite{Sah68}.

(iii)
Under the assumptions of Theorem \ref{thmTF}, we also have the trace formula
\[\tr\big(f(H)-f(H_0)\big)=\frac{1}{\pi i}\int_\bbR f'(\lambda)\zeta(\lambda)\,d\la,\]
where $f$ is a rational function with poles in $\bbC_+$. This follows from the formula \eqref{step1}, which also appeared in \cite[Theorem~6.6]{NM}.

(iv)
Since the functions $\pi\xi(\la)$ and $\zeta(\la)$
are harmonic conjugates of each other, one can avoid appearance of the $A$-integral in the trace formula \eqref{trfimpr} using the equality
$$(A)\int_{\bbR} f'(\lambda)\xi(\lambda)\,d\lambda
=-\frac{1}{\pi } \int_{\bbR} (\mathcal{T}f')(\lambda)\zeta(\lambda)\,d\lambda,
$$
with $\mathcal{T}$ the Hilbert transform and standard Lebesgue integral on the right-hand side.

(v)
The trace formula \eqref{trfimpr} is an accumulative analog of a regularized trace formula obtained by A. Rybkin in \cite{R94} for  contractive trace class perturbations of a unitary operator. However, it is worth mentioning that Rybkin's approach requires a concept of a spectral shift distribution and invokes B-integration in the corresponding trace formula.
\end{remark}

\section{Non-integrability of the Spectral Shift Function}

In this  concluding section we discuss two important examples that emphasize some properties of the spectral shift function  that are not available in the standard trace class perturbation theory for self-adjoint operators.

We start with the observation that since $\xi$ is the Hilbert transform of an integrable function, one automatically has that $\xi\in L^1_w(\bbR)$, the weak $L^1$ space. However, the following example shows that $\xi\notin L^1_{w,0}(\bbR)\subset L^1_w(\bbR)$.

\begin{example}(cf. \cite[Ex.~3.6]{N88}) \lb{notL1}
Let $H_0=0$ and $H=-\alpha i P$, where $\alpha>0$ and $P$ is a rank one orthogonal projection. The function $\xi$ for the pair $H$ and $H_0$ can be computed explicitly
\begin{align}\label{xieq}
\xi(\lambda)=\frac1\pi
\lim_{\varepsilon\to 0^+} \Im \left
 (\log (1+i\alpha(\lambda+i\varepsilon)^{-1})\right )
= \frac1\pi \arctan \left (\frac{\alpha}{\lambda}\right ),
\end{align}
and hence, $\xi\notin L^1_{w,0}(\bbR)$ since $\arctan(a/\la)\sim a/\la$ as $\la\to\infty$. Note that the function $\zeta$ from Lemma \ref{l1} is given by
\begin{equation}\label{zetaeq}
\zeta(\lambda)=
\lim_{\varepsilon\to 0^+} \Re \left
 (\log (1+i\alpha(\lambda+i\varepsilon)^{-1})\right )
= \log\sqrt{1+\frac{\alpha^2}{\lambda^2}},
\end{equation}
and, therefore, $\zeta\in L^1(\bbR)$.
\end{example}

In fact, the phenomenon of $\xi\notin L^1_{w,0}(\bbR)$ observed in Example \ref{notL1} is of general character.
As Theorem \ref{notl1} below shows, the spectral shift function $\xi(\lambda)=\xi(\lambda, H_0, H)$, being the Hilbert transform of a nonnegative integrable function $\zeta(\lambda)$, is never an element of $ L^1_{w,0}(\bbR)$, unless the operator $H$ is also self-adjoint. A weaker statement that in the context of an accumulative perturbation the spectral shift function $\xi$ is necessarily not in $L^1(\bbR)$, follows from  the claim  in \cite[6.1, p.~48]{Stein} that the Hilbert transform of a positive $L^1(\bbR)$ function is not in $L^1(\bbR)$.

\begin{theorem}\label{notl1}
If $f\in L^1(\bbR)$ is such that $\int_\bbR f(y)\,dy\neq 0$, then the Hilbert transform of $f$,
\begin{align}
g(x)=\text{\rm p.v.}\!\!\int_\bbR\f{f(y)}{y-x}\,dy,
\end{align}
is not in $L^1_{w,0}(\bbR)$, and in particular, not integrable.
\end{theorem}

\begin{proof}
For any $h\in L^1(\bbR)$ it follows from the Dominated Convergence Theorem that
\begin{align}\lb{4.5}
\limsup_{t\to 0^+}t\big|\big\{x:\, |h(x)|>t\big\}\big|
&=\limsup_{t\to 0^+}\int_{\{x:\, |h(x)|>t\}}t\,dx \no \\
&\leq \limsup_{t\to 0^+}\int_\bbR\min\{t,|h(x)|\}\,dx = 0.
\end{align}
Thus, it suffices to show that $g$ does not satisfy \eqref{4.5}. In fact, we will show that
\begin{align}
\limsup_{t\to 0^+}\,t\,\big|\big\{x:\, |g(x)|>t\big\}\big| \geq 2\left|\int_\bbR f(y)\,dy\right|>0.
\end{align}
In the following, we split $g$ into three parts
\begin{align}\lb{g012}
g(x) &= \f{-1}{x}\int_{-M}^M f(y)\,dy + \text{\rm p.v.}\!\!\int_{-M}^M \f{yf(y)}{x(y-x)}\,dy + \text{\rm p.v.}\!\!\int_{|y|>M} \f{f(y)}{y-x}\,dy \no
\\
&=: g_{0,M}(x)+g_{1,M}(x)+g_{2,M}(x).
\end{align}
For any $0<\eps<1/2$ and $t>0$, the inequality $|g_{0,M}(x)|=|g(x)-g_{1,M}(x)-g_{2,M}(x)|>t$ implies that either $|g(x)|>(1-2\eps)t$ or $|g_{1,M}(x)|>\eps t$ or else $|g_{2,M}(x)|>\eps t$. Hence,
\begin{align}\lb{RTrIneq}
|\{x:|g(x)|>(1-2\eps)t|\}| &\geq |\{x:|g_{0,M}(x)|>t\}| \no\\
&\quad - |\{x:|g_{1,M}(x)|>\eps t\}| -|\{x:|g_{2,M}(x)|>\eps t\}|.
\end{align}
Since the function $g_{0,M}(x)$ is a constant multiple of $1/x$, we compute
\begin{align}\lb{g0est}
\lim_{M\to\infty}\limsup_{t\to 0^+}\,t\,|\{x:|g_{0,M}(x)|>t\}| = \lim_{M\to\infty}2\left|\int_{-M}^M f(y)\,dy\right| = 2\left|\int_\bbR f(y)\,dy\right|.
\end{align}
Using the inequality $\left|\int_{-M}^M\f{yf(y)}{x(y-x)}\,dy\right|\leq \f{2M\|f\|_1}{|x|^2}$ for all $|x|>2M$, we estimate
\begin{align}\lb{g1est}
\limsup_{t\to 0^+}\,t\,|\{x:|g_{1,M}(x)|>t\}| \leq \limsup_{t\to 0^+}\, 2t\,\left(2M+\sqrt{\f{2M\|f\|_1}{t}}\,\right) = 0.
\end{align}
Denoting by $C$ the norm of the Hilbert transform as a map from $L^1(\bbR)$ to $L^1_w(\bbR)$, we obtain
\begin{align}\lb{g2est}
\lim_{M\to\infty}\sup_{t>0}\,t\,|\{x:|g_{2,M}(x)|>t\}| \leq \lim_{M\to\infty}C\int_{|y|>M}|f(y)|\,dy = 0.
\end{align}
Finally, combining the above estimates \eqref{RTrIneq}--\eqref{g2est} implies
\begin{align}
\limsup_{t\to 0^+}\,t\,|\{x:|g(x)|>t\}| &= \limsup_{t\to 0^+}\,(1-2\eps)\,t\,|\{x:|g(x)|>(1-2\eps)t\}| \no\\
&\geq(1-2\eps)\lim_{M\to\infty}\limsup_{t\to 0^+}\Big[ t\,|\{x:|g_{0,M}(x)|>t\}| \no\\
&\quad - t\,|\{x:|g_{1,M}(x)|>\eps t\}| - t\,|\{x:|g_{2,M}(x)|>\eps t\}| \Big] \no\\
&\geq (1-2\eps)2\left|\int_\bbR f(y)\,dy\right|.
\end{align}
Since $g$ does not satisfy \eqref{4.5}, it is not in $L^1_{w,0}(\bbR)$ and hence not in $L^1(\bbR)$.
\end{proof}

\begin{remark}
It follows from the proofs of Theorem~\ref{logdet} and Theorem~\ref{notl1} that as long as the perturbation $V$ is not zero, the spectral shift function $\xi(\la)$ satisfies
\begin{align}
\limsup_{t\to\infty}\,t|\{\la:\,|\xi(\la)|>t\}| = 0 \quad\text{and}\quad
\limsup_{t\to0^+}\,t|\{\la:\,|\xi(\la)|>t\}| > 0.
\end{align}
\end{remark}

Our second example shows that the spectral shift function does not even need to be locally integrable.
\begin{example}(cf. \cite[Ex.~3.10]{N88}) \lb{notL1loc}
Let $H_0=0$ and $H=-i\sum_{n=1}^\infty\al_n P_n$, where $\{\al_n\}_{n=1}^\infty$ is a summable sequence of positive numbers so that $\sum_{n=1}^\infty \al_n\ln(\al_n)$ is divergent and $\{P_n\}_{n=1}^\infty$ is a sequence of rank one orthogonal projections such that $P_nP_k=0$ whenever $n\neq k$. As in the previous example, the functions $\xi$ and $\zeta$ for the pair $H$ and $H_0$ can be computed explicitly
\begin{align}
\xi(\lambda)= \frac1\pi \sum_{n=1}^\infty \arctan \left(\frac{\al_n}{\lambda}\right) \quad\text{and}\quad
\zeta(\lambda)= \sum_{n=1}^\infty\log\sqrt{1+\frac{\al_n^2}{\lambda^2}}.
\end{align}
Since $\int_0^1\arctan(\frac{\al_n}{\la})d\la = \frac{\al_n}{2}\ln(1+\al_n^2)+\arctan(\al_n)-\al_n\ln(\al_n)$, it follows from the monotone convergence theorem and the divergence of $\sum_{n=1}^\infty \al_n\ln(\al_n)$ that $\int_0^1 |\xi(\la)|d\la=\infty$. Hence, $\xi$ is not locally integrable and, in particular, not in $L^1(\bbR;\frac{d\la}{1+\la^2})$. On the other hand, since $\int_\bbR \log\sqrt{1+\frac{\al_n^2}{\lambda^2}}d\la = \al_n\int_\bbR \log\sqrt{1+\frac{1}{\lambda^2}}d\la$, it follows from the monotone convergence theorem and the summability assumption on $\{\al_n\}_{n=1}^\infty$ that $\zeta\in L^1(\bbR)$.
\end{example}


\end{document}